\documentclass[12pt]{amsart}

\usepackage{graphicx,amssymb,latexsym,amsfonts,txfonts,amsmath,amsthm}
\usepackage{pdfsync,color,tabularx,rotating}
\usepackage[all,cmtip]{xy}
\usepackage{url}
\usepackage{tikz}
\usepackage{amssymb}
\usepackage{comment}

\advance\hoffset by -0.9 truecm

\newtheorem{theo}{Theorem}[section]

\newtheorem{lemm}[theo]{Lemma}
\newtheorem{conj}[theo]{Conjecture}

\theoremstyle{remark}

\theoremstyle{remark}

\theoremstyle{remark}

\theoremstyle{remark}

\begin{document}

\title{Counting conjugacy classes of subgroups of ${\rm PSL}_2(p)$}

\author{Gareth A.~Jones}

\address{School of Mathematical Sciences, University of Southampton, Southampton SO17 1BJ, UK}
\email{G.A.Jones@maths.soton.ac.uk}




\begin{abstract}
We obtain formulae for the numbers of isomorphism and conjugacy classes of non-identity proper subgroups of the groups $G={\rm PSL}_2(p)$, $p$ prime, and for the numbers of those conjugacy classes which do or do not consist of self-normalising subgroups. The formulae are used to prove lower bounds $17$, $18$, $6$ and $12$ respectively satisfied by these invariants for all $p>37$. A computer search carried out for a different problem shows that these bounds are attained for over a million primes $p$; we show that if the Bateman--Horn Conjecture is true, they are attained for infinitely many primes. Also, assuming no unproved conjectures, we use a result of Heath-Brown to obtain upper bounds for these invariants, valid for an infinite set of primes $p$.
\end{abstract}

\maketitle


\section{Introduction}\label{sec:intro}

The aim of this note is to consider the following invariants of a finite group $G$:

\begin{itemize}
\item[(i)] $i(G)$, the number of isomorphism classes of non-identity proper subgroups $H$ of $G$;
\item[(ii)] $c(G)$,  the number of conjugacy classes of such subgroups $H$ of $G$;
\item[(iii)] $s(G)$,  the number of conjugacy classes of such self-normalising subgroups $H$ of $G$;
\item[(iv)] $n(G)$,  the number of conjugacy classes of such non-self-normalising subgroups $H$ of $G$.
\end{itemize}
A subgroup $H\le G$ is {\em self-normalising} if $H$ coincides with its normaliser $N_G(H)$. The motivation for this paper was a seminar by Mark Lewis~\cite{Lew} on $n(G)$; this was mainly about solvable groups, but it raised a question about non-abelian finite simple groups, answered in Theorems~\ref{thm:ifBHC} and \ref{thm:H-B}.

Using Dickson's description~\cite[Ch.~XII]{Dickson} of the subgroups of ${\rm PSL}_2(q)$, where $q$ is a prime power (see also~\cite[\S II.8]{Hup}) we will give formulae for these invariants for the groups $G={\rm PSL}_2(p)$ ($p$ prime). As a consequence, we will prove the following:
\
\begin{theo}\label{thm:bounds}
The groups $G={\rm PSL}_2(p)$, $p$ prime, satisfy
\begin{itemize}
\item[{\rm(i)}] $i(G)\ge 17$ for all $p\ge 29$;
\item[{\rm(ii)}] $c(G)\ge 18$ for all $p\ge 23$;
\item[{\rm(iii)}] $s(G)\ge 6$ for all $p\ge 41$;
\item[{\rm(iv)}] $n(G)\ge 12$ for all $p\ge 23$.
\end{itemize}
\end{theo} 

The exceptions for small $p$ can be seen in Table~\ref{tab:small-p}, which lists the values of these invariants for primes $p=3,\ldots, 61$. A computer search reported in~\cite{JZ}, in relation to a different problem concerning these groups $G$, found a set of more than $1.2\times 10^6$ primes $p$ for which all four lower bounds in Theorem~\ref{thm:bounds} are attained. In fact, we make the following conjecture:

\begin{conj}\label{conj:infinitely-many}
There is an infinite set of primes $p$ for which the groups $G={\rm PSL}_2(p)$ satisfy $i(G)=17$, $c(G)=18$, $s(G)=6$ and $n(G)=12$.
\end{conj}

A proof of this conjecture would depend on a certain triple of polynomials in ${\mathbb Z}[t]$ simultaneously taking prime values for infinitely many $t\in{\mathbb N}$. This sort of problem is the subject of the Bateman--Horn Conjecture (BHC), which gives estimates for the number of $t\le x$, for large $x$, giving prime values to a given finite set of polynomials. In all {\color{cyan}known} applications these estimates agree closely with data from computer searches, so there is strong evidence (but at present no proof) that the BHC is true. Our second main result is the following:

\begin{theo}\label{thm:ifBHC}
If the BHC is true then there is an infinite set of primes $p$ for which the lower bounds $17$, $18$, $6$ and $12$ in Theorem~\ref{thm:bounds} are attained.
\end{theo}

In the absence of the BHC and any other unproved conjectures, we are able to use a result of Heath-Brown~\cite{BLS} to prove the following:

\begin{theo}\label{thm:H-B}
There is an infinite set of primes $p$ for which
the groups $G={\rm PSL}_2(p)$ satisfy
\begin{itemize}
 \item[{\rm(i)}] $i(G)\le 390$;
 \item[{\rm(ii)}] $c(G)\le 454$;
 \item[{\rm(iii)}] $s(G)\le 132$;
 \item[{\rm(iv)}] $n(G)\le 384$.
 \end{itemize}
\end{theo}

The existence of an infinite set of non-abelian finite simple groups $G$ for which $n(G)$ is bounded above answers a question raised by Lewis in~\cite{Lew}.
The large gaps between the lower and upper bounds in Theorems~\ref{thm:bounds} and \ref{thm:H-B} suggest that there is further work to be done in reducing these upper bounds without assuming the BHC. In the meantime, work is in progress on an analogue of Theorem~\ref{thm:bounds} for the groups ${\rm PSL}_2(q)$ where $q$ is any prime power; the powers of primes $p\ge 5$ are straightforward, whereas the cases $p=2$ and $3$ raise further interesting number-theoretic problems.


\section{Subgroups of ${\rm PSL}_2(p)$}\label{sec:subgroups}

By Dickson's description~\cite[Ch.~XII]{Dickson} of the subgroups of ${\rm PSL}_2(q)$, we have the following:

\begin{lemm}\label{lemma:isotypes}
For any prime $p>2$ the isomorphism types of non-identity proper subgroups of $G:={\rm PSL}_2(p)$ are as follows, where $\delta$ and $\varepsilon$ are the numbers of divisors $d$ and $e$ of $(p\pm1)/2$ respectively: .
\begin{enumerate}
\item ${\rm C}_d$ and ${\rm D}_d$ for the divisors $d\ne 1$ of $(p+1)/2$, giving $\delta-1$ types in either case;
\item  ${\rm C}_e$ and ${\rm D}_e$ for the divisors $e\ne 1$ of $(p-1)/2$, giving $\varepsilon-1$ types in either case; 
\item ${\rm C}_p\rtimes{\rm C}_e$ for the divisors $e$ of $(p-1)/2$, giving $\varepsilon$ types;
\item ${\rm A}_4$ if $p\ge 5$;
\item ${\rm S}_4$ if $p\equiv \pm 1$ {\rm mod}~$(8)$;
\item ${\rm A}_5$ if $p\equiv\pm 1$ {\rm mod}~$(5)$.
\end{enumerate}
\end{lemm}

We will assume that $p\ge 5$. Adding up the above numbers, we see that the total number $i(G)$ of such isomorphism types is given by
\begin{equation}\label{eq:i-formula}
i(G)=2\delta+3\varepsilon-3+\sigma+\alpha,
\end{equation}
where $\sigma=1$ or $0$ as $p\equiv\pm 1$ or $\pm 3$ mod~$(8)$, and $\alpha=1$ or $0$ as $p\equiv\pm 1$ mod~$(5)$ or not.

An integer $n=p_1^{e_1}\ldots p_k^{e_k}$ ($p_i$ distinct primes) has $\tau(n)=(e_1+1)\ldots(e_k+1)$ divisors. Since the average value of $\tau(n)$ is approximately $\log n$ (see, for instance, \cite[Theorem 6.30]{LeV}), if we choose the prime $p$ randomly, we should expect $(p\pm 1)/2$ each to have on average about $\log\,(p/2)$ divisors. It follows from de la Vall\'ee-Poussin's quantified form of Dirichlet's Theorem that $\sigma$ and $\alpha$ each have average value $1/2$, so on average we have
\[i(G)\approx 5\log p-2.\]
However, we aim to choose primes $p$ for which $(p\pm 1)/2$ each have a small number of divisors, so that $i(G)$ and the related invariants $c(G)$, $s(G)$ and $n(G)$ are as small as possible; more specifically, we want an infinite set of primes for which these invariants are uniformly bounded above. To see how small we can make these upper bounds, we will in the next section first look for lower bounds on these invariants. In order to do this, we will now find formulae analogous to (\ref{eq:i-formula}) for the remaining invariants.

By inspection of~\cite[Ch.~XII]{Dickson}, we have the following:

\begin{lemm}\label{lemma:2classes}
Each of the isomorphism types in Lemma~\ref{lemma:isotypes} is represented by a single conjugacy class in $G$, with the following exceptions, each represented by two classes:
\begin{itemize}
\item[(1)] ${\rm D}_d$ for those $d\ne 1$ dividing $(p+1)/2$ such that $(p+1)/2d$ is even;
\item[(2)] ${\rm D}_e$ for those $e\ne 1$ dividing $(p-1)/2$ such that $(p-1)/2e$ is even;
\item[(4)] ${\rm A}_4$ if $p\equiv\pm 1$ {\rm mod}~$(8)$;
\item[(5)] ${\rm S}_4$ if $p\equiv\pm 1$ {\rm mod}~$(8)$;
\item[(6)] ${\rm A}_5$ if $p\equiv\pm 1$ {\rm mod}~($5)$.
\end{itemize}
\end{lemm}

The proportion of the divisors $d$ or $e$ of $(p\pm 1)/2$ with an even quotient is $k/(k+1)$ or $l/(l+1)$ where $2^k$ and $2^l$ are the highest powers of $2$ dividing $(p\pm 1)/2$. Since $(p\pm 1)/2$ have opposite parity, either $k=0$ or $l=0$, as $p\equiv \pm 1$ mod~$(4)$. In just one of these two cases the divisor $1$ with even quotient must be excluded, so the number $c(G)$ of conjugacy classes of non-identity proper subgroups of $G$ is given by
\begin{equation}\label{eq:c-formula}
\begin{split}
c(G)&=i(G)+\frac{k}{k+1}\delta+\frac{l}{l+1}\varepsilon-1+2\sigma+\alpha\cr
&
=\left(2+\frac{k}{k+1}\right)\delta+\left(3+\frac{l}{l+1}\right)\varepsilon-4+3\sigma+2\alpha.
\end{split}
\end{equation}

\medskip
 
 A further inspection of~\cite[Ch.~XII]{Dickson} gives the following:
 
 \begin{lemm}\label{lemma:self-norm}
The only isomorphism types in Lemma~\ref{lemma:isotypes} which consist of self-normalising subgroups of $G$ are the following:
\begin{enumerate}
 \item ${\rm D}_d$ for divisors $d>2$ of $(p+1)/2$ with $(p+1)/2d$ odd, one class for each $d$ (a subgroup ${\rm D}_2={\rm V}_4$ is never self-normalising);
  \item ${\rm D}_e$ for divisors $e>2$ of $(p-1)/2$ with $(p-1)/2e$ odd, one class for each $e$;
 \item ${\rm C}_p\rtimes{\rm C}_{(p-1)/2}$, one class;
 \item ${\rm A}_4$ if $p\equiv\pm 3$ {\rm mod}~$(8)$, one class;
 \item ${\rm S}_4$ if $p\equiv\pm 1$ {\rm mod}~$(8)$, two classes;
 \item ${\rm A}_5$ if $p\equiv\pm 1$ {\rm mod}~$(5)$, two classes.
 \end{enumerate}
 \end{lemm}
 
The subgroups of types (3), (4), (5) and (6) are maximal in $G$, as are the subgroups ${\rm D}_{(p\pm 1)/2}$ in (1) and (2); the remaining subgroups listed above  in (1) and (2) are self-normalising proper subgroups of ${\rm D}_{(p\pm 1)/2}$. One of $(p\pm 1)/2$ is odd, so each of its $\delta-1$ or $\varepsilon-1$ divisors $d\ne 1$ or $e\ne 1$ yields a class of self-normalising dihedral subgroups. For the even integer $(p\pm 1)/2$, the proportion of its divisors which have an odd quotient is $1/(k+1)$ or $1/(l+1)$; such divisors do not include $1$, and they include $2$ if and only if $k=1$ or $l=1$, that is, $p\equiv\pm 5$ mod~$(8)$, or equivalently $\sigma=0$. It follows that the total number of self-normalising classes in (1) and (2) can be written as
\[\frac{\delta}{k+1}+\frac{\varepsilon}{l+1}-2+\sigma.\]
(Recall that one of $k$ and $l$ is zero, as $p\pm 1)/2$ is odd.)  In addition, (3)  contributes one class, (4) and (5) together contribute $\sigma+1$, and (6) contributes $2\alpha$. Adding all these contributions we see that the total number $s(G)$ of conjugacy classes of self-normalising subgroups is given by
\begin{equation}\label{eq:s-formula}
s(G)=\frac{\delta}{k+1}+\frac{\varepsilon}{l+1}+2(\sigma+\alpha).
\end{equation}
It follows that the number $n(G)$ of conjugacy classes of non-self-normalising subgroups of $G$ is given by
\begin{equation}\label{eq:n-formula}
n(G) = c(G) - s(G)
= \left(2+\frac{k-1}{k+1}\right)\delta + \left(3+\frac{l-1}{l+1}\right)\varepsilon
-4+\sigma.
\end{equation}

\noindent{\bf Example 2.4}\, Let $p=37$. Then $(p+1)/2=19$, which is prime, so $\delta=2$, and $(p-1)/2=18=2\cdot 3^2$, so $\varepsilon=2\cdot 3=6$. We have $k=0$ and $l=1$. Since $p\equiv -3$ mod~$(8)$ and $p\equiv 2$ mod~$(5)$ we have $\sigma=\alpha=0$. Thus
\[i(G)=2\cdot2+3\cdot 6-3=19,\]
\[c(G)=2\cdot 2+\frac{7}{2}\cdot 6-4=21\]
with subgroups ${\rm D}_3$ and ${\rm D}_9$ each forming two conjugacy classes,
\[s(G)=\frac{2}{1}+\frac{6}{2}=5\]
represented by conjugacy classes of subgroups ${\rm D}_{19}$, ${\rm D}_6$, ${\rm D}_{18}$, ${\rm C}_{37}\rtimes{\rm C}_{18}$, and ${\rm A}_4$ (one each), and
\[n(G)=1\cdot 2+3\cdot 6-4=16\]
represented by conjugacy classes of subgroups ${\rm C}_{19}$, ${\rm C}_e$ ($e=2, 3, 6, 9, 18$), ${\rm D}_2$, ${\rm D}_3$ (two classes), ${\rm D}_9$ (two classes) and ${\rm C}_{37}\rtimes{\rm C}_e$ ($e=1, 2, 3, 6, 9$). Table~\ref{tab:small-p} gives similar data for the primes $p=3,\ldots, 61$.

\begin{table}[htbp]
\begin{center}
\begin{tabular}{c||c|c|c|c|c|c||c|c|c|c}
$p$ & $\delta$ & $\varepsilon$ & $k$ & $l$ & $\sigma$ & $\alpha$ & $i(G)$ & $c(G)$ & $s(G)$ & $n(G)$  \\
\hline
3    & 1 & 2 & 0 & 1 & 0 & 0 & 3   & 3   & 1 &  2 \\
5    & 2 & 2 & 0 & 1 & 0 & 0 & 7   & 7   & 3 &  4 \\
7    & 3 & 2 & 2 & 0 & 1 & 0 & 10 & 14 & 5 &  8 \\
11  & 4 & 2 & 1 & 0 & 0 & 1 & 12 & 14 & 6 &  8 \\
13  & 2 & 4 & 0	& 1 & 0 & 0 & 13 & 14 & 4 & 10	\\
17  & 3 & 4 & 0 & 3 & 1 & 0 & 16 & 20 & 6 & 14 \\
19  & 4 & 3 & 1	& 0 & 0 & 1 & 15 & 17 & 7 & 10	\\
23  & 6 & 2 & 2 & 0 & 1 & 0 & 16 & 21 & 6 & 15 \\
29  & 4 & 4 & 0 	& 1 & 0 & 1 & 18 & 20 & 8 & 12 \\
31  & 5 & 4 & 4 & 0 & 1 & 1 & 21 & 27 & 9 & 18	\\
37  & 2 & 6 & 0 	& 1 & 0 & 0 & 19 & 21 & 5 & 16 \\
41  & 4 & 6 & 0 & 2 & 1 & 1 & 25 & 31 & 10 & 21 \\
43  & 4 & 4 & 1 & 0 & 0 & 0 & 17 & 18 & 6 & 12 \\
47  & 8 & 2 & 3 & 0 & 1 & 0.& 20 & 27 & 6 & 21 \\
53  & 4 & 4 & 0 & 1 & 0 & 0 & 17 & 18 & 6 & 12 \\
59  & 8 & 2 & 1 & 0 & 0 & 1 & 20 & 24 & 8 & 16 \\
61  & 2 & 8 & 0 & 1 & 0 & 1 & 26 & 30 & 8 & 22
\end{tabular}
\end{center}
\vspace{2mm}
\caption{Invariants of ${\rm PSL}_2(p)$ for small primes $p$.}
\label{tab:small-p}
\end{table}


\section{Lower bounds for the invariants}

Our aim now is to find primes $p$ for which the invariants considered above for the groups $G={\rm PSL}_2(p)$ take the lowest possible values. In order to obtain uniform bounds we will assume that $p>37$; the smaller primes can be dealt with individually, as in the preceding example (see Table~\ref{tab:small-p}).

We start with $i(G)$, given by (\ref{eq:i-formula}). The first step is to ensure that $\sigma=\alpha=0$ by restricting attention to primes $p\equiv\pm 3$ or $\pm 13$ mod~$(40)$; Dirichlet's Theorem guarantees an infinite set of primes satisfying any one of these four congruences. It now remains to find a good upper bound on $2\delta+3\varepsilon$, and this can be done by minimising the numbers of primes dividing $(p\pm 1)/2$. These two consecutive integers are coprime, one of them must be divisible by $2$, and one of them by $3$, giving four possibilities for allocating these prime factors. They are as follows, with inequalities needed to produce primes $p>37$:

\begin{itemize}
\item[(a)] $(p+1)/2=3s$ and $(p-1)/2=2r$ coprime, with $s\ge 7$ and $r\ge 10$,
\item[(b)] $(p+1)/2=2s$ and $(p-1)/2=3r$ coprime, with $s\ge 10$ and $r\ge 7$,
\item[(c)] $(p+1)/2=6s$ and $(p-1)/2=r$ coprime, with $s\ge 4$ and $r\ge 19$,
\item[(d)] $(p+1)/2=s$ and $(p-1)/2=6r$ coprime, with $s\ge 20$ and $r\ge 4$,
\end{itemize}
where $|G|=12psr$ in all cases. (Remarkably, these four cases also arose recently in~\cite{JZ}, which addressed the problem of constructing a possibly infinite set of finite simple groups of order a product of six primes.)

If in (a) or (b) we take $s$ and $r$ to be primes satisfying the specified inequalities and giving a prime value for $p$, then $\delta=\varepsilon=4$ and so $2\delta+3\varepsilon=20$; moreover, $|G|$ is not divisible by $8$ or by $5$, so $\sigma=\alpha=0$ and thus $i(G)=17$. This value can also be obtained in (a) by taking $s=9$ and $r=13$, so that $p=53$, but it is not hard to see that no lower value of $i(G)$ can be obtained in cases (a) to (d). Thus
\begin{equation}\label{eq:i-bound}
i(G)\ge 17
\end{equation}
for all $G$ with $p>37$, with equality whenever conditions (a) or (b) are satisfied by primes $p$, $s$ and $r$, and also when $p=53$. In fact, Table~\ref{tab:small-p} shows that this inequality is satisfied for all $p\ge 29$, but not for any smaller primes. This proves Theorem~\ref{thm:bounds}(i).

A similar argument can be applied to $c(G)$, given by (\ref{eq:c-formula}), but now we need to minimise
\[\left(2+\frac{k}{k+1}\right)\delta+\left(3+\frac{l}{l+1}\right)\varepsilon
=\left(3-\frac{1}{k+1}\right)\delta+\left(4-\frac{1}{l+1}\right)\varepsilon,\]
where $k$ and $l$ take the values $0$ and $1$ in either order. Again, cases (a) and (b) achieve the minimum value for this expression, namely $22$, so that
\begin{equation}\label{eq:c-bound}
c(G)\ge 18
\end{equation}
for all $G$ with $p>37$, with the same conditions for equality as in (\ref{eq:i-bound}). In this case, Table~\ref{tab:small-p} shows that the bound is in fact satisfied for all $p\ge 23$, but not for $p=19$. This proves Theorem~\ref{thm:bounds}(ii).

To minimise $s(G)$, given by (\ref{eq:s-formula}), we need to minimise
\[\frac{\delta}{k+1}+\frac{\varepsilon}{l+1},\]
where again $\{k, l\}=\{0, 1\}$. This is achieved in all four cases (a) to (d) whenever one can take $s$ and $r$ prime. The resulting value is $6$, so that
\begin{equation}\label{eq:s-bound}
s(G)\ge 6
\end{equation}
for all $G$ with $p>37$, attained in cases (a) to (d), and thus proving Theorem~\ref{thm:bounds}(iii). However, in Example~2.4, where $p=37$, we have $s(G)=5$: the reason is that $(p-1)/2=18=2\cdot 3^2$ has a repeated prime factor, so that $\varepsilon =6$ rather than $8$ which would be the case for three distinct prime factors. This explains why we have assumed here that $p>37$.

Finally, to minimise $n(G)$ we need to minimise
\[\left(3-\frac{2}{k+1}\right)\delta+\left(4-\frac{2}{l+1}\right)\varepsilon.\]
Once again, cases (a) and (b) give the least value of this expression, namely $16$, so that
\begin{equation}\label{eq:n-bouind}
n(G)\ge 16-4=12
\end{equation}
for all $G$ such that $p>37$, with equality as in (a) and (b). In fact, Table~\ref{tab:small-p} shows that this bound is satisfied for all $p\ge 23$. This proves Theorem~\ref{thm:bounds}(iv).


\section{Attaining the bounds}

We have seen instances of groups $G$ which attain the four lower bounds obtained in the preceding section, namely all the groups in cases (a) and (b) involving triples of primes $p$, $r$ and $s$. The smallest such group in case (a) is ${\rm PSL}_2(29)$; however, this has $\alpha=1$ since its order is divisible by $s=5$, so it does not satisfy all the bounds. The next example in case (a) is ${\rm PSL}_2(173)$, with $r=43$ and $s=29$, and this has $\alpha=0$, as required. The first example in case (b) is ${\rm PSL}_2(43)$, with $r=7$ and  $s=11$. In fact, in each of cases (a) to (d) the computer searches reported in~\cite{JZ} found over $600,000$ instances of suitable triples of primes, all but a few having $\sigma=\alpha=0$, so that those in cases (a) and (b) attain equality in our bounds.

Ideally, we would like infinite sets of examples attaining these bounds.
In case~(a) the equations
\[(p+1)/2=3s\quad\hbox{and}\quad (p-1)/2=2r\]
imply that $p\equiv 5$ mod~$(12)$, so we can write $p=12t+5$ for some $t\in{\mathbb N}$. Then
\[s=(p+1)/6=2t+1\quad\hbox{and}\quad r=(p-1)/4=3t+1.\]
We require $p$, $s$ and $r$ all to be prime, so we need the the polynomials
\begin{equation}\label{eq:a-polys}
f_i(t)=12t+5,\quad 3t+1\quad\hbox{and}\quad 2t+1,
\end{equation}
representing them all to take prime values for infinitely many $t\in{\mathbb N}$. 

The situation is similar in cases (b), (c) and (d): for example, in case~(b) we need the polynomials 
\begin{equation}\label{eq:b-polys}
f_i(t)=12t+7,\quad 3t+2\quad \hbox{and}\quad 2t+1.
\end{equation}
to take prime values for infinitely many $t\in{\mathbb N}$. This type of problem is the subject-matter of Schinzel's Hypothesis and the Bateman--Horn Conjecture, considered in the next section.


\section{The Bateman--Horn Conjecture}

Given a finite set of distinct polynomials $f_i(t)\in{\mathbb Z}[t]$ ($i=1,\ldots, k$), it is of interest to know whether they simultaneously take prime values for infinitely many $t\in{\mathbb N}$.
Instances of this general problem include the following:
\begin{itemize}
\item the twin primes problem, with $f_1(t)=t$ and $f_2(t)=t+2$;
\item the Sophie Germain primes problem, with $f_1(t)=t$ and $f_2(t)=2t+1$;
\item the Euler--Landau problem, with a single polynomial $f_1(t)=t^2+1$.
\end{itemize}
The following are obvious necessary conditions on the polynomials $f_i(t)$:
\begin{itemize}
\item they all have positive leading coefficients;
\item they are all irreducible;
\item their product $f(t)=f_1(t)\ldots f_k(t)$ is not identically zero modulo any prime.
\end{itemize}
Schinzel's Hypothesis~\cite{SS}, which we will abbreviate to SH, asserts that these conditions are also sufficient. This generalises earlier conjectures by Bunyakovsky~\cite{Bun} for $k=1$ and by Dickson~\cite{Dickson04} where $\deg f_i=1$ for all $i$. It has been proved only in the simplest case, where $k=1$ and $\deg f_1=1$: this is Dirichlet's Theorem on primes in an arithmetic progression $at+b$.

Building on earlier work by Hardy and Littlewood~\cite{HL} on some special cases, the Bateman--Horn Conjecture (BHC) quantifies SH by giving an estimate $E(x)$ for the number $Q(x)$ of $t\le x$ for which the values $f_i(t)$ are all prime. It asserts that
\[Q(x)\sim E(x)\quad\hbox{as}\quad x\to\infty\]
where
\begin{equation}\label{eq:BHC-est}
E(x):=C\int_a^{+\infty}\negthickspace\frac{dt}{\prod_i\ln f_i(t)};
\end{equation}
here $C$ is the {\sl Hardy--Littlewood constant}
\begin{equation}\label{eq:C-defn}
C=C(f_1,\ldots, f_k):=\prod_{p\;{\rm prime}}\left(1-\frac{1}{p}\right)^{-k}\left(1-\frac{\omega_f(p)}{p}\right)
\end{equation}
where $\omega_f(p)$ is the number of distinct roots of the polynomial $f=f_1\ldots f_k$ mod~$(p)$, and $a$ is chosen so that the integral in (\ref{eq:BHC-est}) avoids any singularities, where $\ln f_i(t)=0$ for some $i$.

One can regard the integral in (\ref{eq:BHC-est}) as a first attempt at an estimate for $Q(x)$, based on applying the Prime Number Theorem to the polynomials $f_i(t)$, on the (incorrect) assumption that they take values independently. The constant $C$ in (\ref{eq:C-defn}) can then be regarded as a product of correction factors, one for each prime $p$, replacing the probability $(1-\frac{1}{p})^k$ that $k$ randomly and independently chosen elements of ${\mathbb Z}_p$ are all non-zero with the proportion of ${\mathbb Z}_p$ where $f(t)\ne 0$.

When the three conditions of SH are satisfied the infinite product in (\ref{eq:C-defn}) converges to a limit $C>0$. Convergence is slow, but good approximations can be found by taking partial products over large initial segments of the primes. On a laptop this can take several hours, but by contrast Maple can accurately evaluate the integral in (\ref{eq:BHC-est}) almost instantly, by numerical integration.

Like SH, the BHC has been proved only in the simplest case, where $k=1$ and $\deg f_1=1$. This is de la Vall\'ee-Poussin's quantified form of Dirichlet's Theorem, that the primes $at+b$ are asymptotically evenly distributed among the $\varphi(a)$ congruence classes of units mod~$(a)$.

If the conditions of SH are satisfied then $E(x)\to+\infty$ as $x\to\infty$. The BHC asserts that $Q(x)\sim E(x)$ as $x\to\infty$, so if the BHC is correct then $Q(x)\to+\infty$ also, and hence there are infinitely many $t\in{\mathbb N}$ for which each $f_i(t)$ is prime. Thus the BHC implies SH. The fact that the BHC has given impressively accurate estimates in all known applications is strong evidence that it is true, but at present there seems to be little prospect of a proof.

There is an interesting account in~\cite{AFG} of the history of the BHC, including a proof of the convergence of the product (\ref{eq:C-defn}). There is a shorter discussion of the BHC in~\cite{JZ}, including several applications to other problems in group theory, some involving the groups ${\rm PSL}_2(p)$.

Returning to our specific problem, the polynomials $f_i(t)$ in case~(a), shown in (\ref{eq:a-polys}), satisfy the common conditions of SH and the BHC, as do those in case~(b) shown in (\ref{eq:b-polys}), so if the BHC is true then in either case there are infinitely many $t\in{\mathbb N}$ with each $f_i(t)$ prime, as required. This gives a proof of Theorem~\ref{thm:ifBHC} and hence (assuming the BHC) of Conjecture~\ref{conj:infinitely-many}.

\medskip

\noindent{\bf Example 5.1}\, As in~\cite{JZ}, where cases (a) to (d) arose in connection with a different problem concerning the groups ${\rm PSL}_2(p)$, let  $x=10^9$. A computer search reported there, using the primality test within Maple, shows that $Q_a(10^9)=614\,423$. Evaluating $E_a(10^9)$ by Maple, using numerical integration, gives an estimate $615\,580.7$, which has a relative error of $+0.188\%$. Similarly, in case~(b) we have $Q_b(10^9)=615\,369$ and $E_b(10^9)=615\,580.6$, an error of $+0.034\%$.


\section{Applying Heath-Brown's result}

What can be proved without assuming any unproved conjectures? Heath-Brown, in an appendix to~\cite{BLS}, has proved that there is an infinite set $\mathcal P$ of primes $p\equiv 5$ mod~$(72)$ such that
\[\Omega(p-1)+\Omega(p+1)\le 11\quad\hbox{and}\quad \Omega(p-1), \Omega(p+1)\le 8,\]
where $\Omega(n)$ is the number $\sum_{i=1}^ke_i$ of primes dividing an integer $n=\prod_{i=1}^kp_i^{e_i}$, counting repetitions.
Note that if $p\in{\mathcal P}$ then $(p-1)/2$ is twice odd and $(p+1)/2$ is odd.

Let ${\mathcal G}=\{G={\rm PSL}_2(p)\mid p\in\mathcal P\}$. If $G\in\mathcal G$ then since $p\equiv 5$ mod~$(8)$ we have $\sigma=0$, so (\ref{eq:i-formula}) becomes
\[i(G)=2\delta+3\varepsilon-3+\alpha.\]
An upper bound (not necessarily attained) on $i(G)$ for $G\in{\mathcal G}$ is therefore obtained by taking $(p+1)/2$ and $(p-1)/2$ to be products of two and seven distinct primes, including $5$ in one case, so that $\delta=2^2=4$, $\varepsilon=2^7=128$ and $\alpha=1$.
This gives an upper bound
\[i(G)\le 2\times 4+3\times 128-3+1=390\]
for all $G\in{\mathcal G}$.

Since $p\equiv 5$ mod~$(8)$ we have $k=0$ and $l=1$, so (\ref{eq:c-formula}) becomes
\[c(G)=2\delta+\frac{7}{2}\varepsilon-4+2\alpha.\]
The same assumptions as before on the primes dividing $(p\pm 1)/2$ therefore give an upper bound
\[c(G)\le 2\times 4+\frac{5}{2}\times 128-4+2=454\]
for all $G\in{\mathcal G}$.
Since $k=0$, $l=1$ and $\sigma=0$, (\ref{eq:s-formula}) becomes
\[s(G)=\delta+\frac{1}{2}\varepsilon+2\alpha.\]
With this changed weighting for $\delta$ and $\varepsilon$, an upper bound for $s(G)$ is now  obtained by assigning seven distinct prime factors to $(p+1)/2$ and two to $(p-1)/2$, again including $5$, so that $\delta=128$, $\varepsilon=4$ and $\alpha=1$, giving an upper bound
\[s(G)\le 128+\frac{1}{2}\times 4+2=132\]
for all $G\in{\mathcal G}$.
Finally (\ref{eq:n-formula}) becomes
\[n(G)=\delta+3\varepsilon-4,\]
so the same assignment of prime factors as for $i(G)$ and $c(G)$ gives an upper bound
\[n(G)\le 4+3\times 128-4=384\]
for all $G\in{\mathcal G}$.
This completes the proof of Theorem~\ref{thm:H-B}.


\bigskip

\centerline{\bf Acknowledgments}

\medskip

The author is grateful to Mark Lewis for the seminar which inspired this work, to Natalia Maslova for organising the workshop in which it took place, and to Alexander Zvonkin for some very helpful remarks on this paper and for the computations reported in~\cite{JZ} and used again here.


\end{document}